\documentclass{amsart}
\usepackage{amssymb} 
\usepackage{graphicx} 
\usepackage{amsmath}
\usepackage{amsfonts}
\usepackage{hyperref}
\usepackage{young}
\usepackage[vcentermath]{youngtab}
\usepackage{cite}
\usepackage{mathtools}
\DeclarePairedDelimiter{\ceil}{\lceil}{\rceil}

\usepackage{xcolor}

\theoremstyle{definition}
\newtheorem{definition}{Definition}
\newtheorem{example}{Example}
\newtheorem{remark}{Remark}
\theoremstyle{plain}

\newtheorem{theorem}{Theorem}
\newtheorem{lemma}{Lemma}
\newtheorem{corollary}{Corollary}
\newtheorem{conjecture}{Conjecture}

\begin{document}

	 \title{Degree Growth of Rational Maps with Laurent Property}
	 \author[M.E.~Can]{Melih Emin Can}
	  \address{\normalfont Faculty of Electrical Engineering and
Computing, Department of Applied
Mathematics, University of Zagreb,
Zagreb, Croatia }
	 \email{mcan@fer.hr}

	 \maketitle
\begin{abstract}
Algebraic entropy measures the exponential growth of the degrees of iterates of rational self-maps. It is introduced by Bellon and Viallet in \cite{BellonandViallet}. We study degree growth for rational maps arising from recurrences with the Laurent property, with particular emphasis on Gale--Robinson recurrences. We develop a direct method for determining the degree growth from the recurrence relation itself, reducing the problem to the study of associated max-linear, or tropical recurrences.

We first obtain a closed formula for the degrees of iterates for a broad class of homogeneous quadratic Laurent recurrences. We then apply this result to the two-term Gale--Robinson family and derive an explicit quadratic growth formula with a periodic correction term. For three-term Gale--Robinson recurrences, we express the algebraic entropy in terms of the associated tropical growth rate sequence. Based on explicit computations, we formulate a conjecture for the general three-term family predicting quadratic growth with a periodic correction. We prove the conjecture for the infinite family with parameters $p=1$, $q=2$, and $r\geq3$. 
\end{abstract}

\section{Introduction}

Let $\mathbb{P}^m$ denote the $m$-dimensional projective space. We write its elements as
$
(x_0:\dots:x_m)
$. Consider a dominant rational map

$$
f\colon\mathbb{P}^m\rightarrow\mathbb{P}^m.
$$
The map $f$ can be represented by homogeneous polynomials 
$p_0,\ldots,p_m$ of the same degree and with no common factor:
$$
(x_0:\dots:x_m)
\mapsto
(p_0(x_0,\ldots,x_m):\dots:p_m(x_0,\ldots,x_m)).
$$
The common degree of these polynomials is called the degree of $f$. For $n\geq1$, we write
$$
d_n=\deg(f^n).
$$
Every rational map from $\mathbb{A}^m$ to itself can be regarded as a rational map from $\mathbb{P}^m$ to $\mathbb{P}^m$ by projectivization (see \cite{Propp}), and hence the notion of degree extends to this case.

The algebraic entropy of $f$, introduced in \cite{BellonandViallet}, is defined by
$$
\mathcal{E}(f)
=
\lim_{n\to\infty}\frac{1}{n}\log d_n.
$$
The existence of this limit follows from the subadditivity of $\log d_n$ and Fekete's lemma.

In this paper, we study rational maps arising from recurrences with the Laurent property. The Laurent phenomenon was introduced by Fomin and Zelevinsky \cite{FandZ}. The Laurent phenomenon occurs in many nonlinear recurrences and provides a useful connection between dynamics of rational maps, cluster algebras, and discrete integrable systems (see \cite{AlmanCuencaHuang, FandZ, hone2006diophantine, hone2007singularity, hone2008integrality, FordyHone2012, Fordy2014, Hone2025}).

The ring of Laurent polynomial in the variables $x_1,\ldots,x_s$ with coefficients in a field $\mathbb{F}$, that is $$R=\mathbb{F}[x_1^{\pm},\ldots,x_s^{\pm}]$$ consists of expressions of the form $$\sum_{k\in \mathbb{Z}}c_kx^k,$$ where $x$ is a variable and only finitely many coefficients $c_k$ are nonzero.
Consider a rational map
$$
f\colon\mathbb{A}^m\rightarrow\mathbb{A}^m
$$
of the form
$$
f(x_0,\ldots,x_{m-1})
=
(x_1,\ldots,x_m),
$$%
where $x_m$ is a rational function of $x_0,\ldots,x_{m-1}$. For $n\geq0 $ the $n^\text{th}$ iterate of the map $f$ is
$$
f^n(x_0,\ldots,x_{m-1})
=
(x_n,\ldots,x_{n+m-1}),
$$
where $\{x_n\}_{n\geq0}$ is a sequence of rational functions. We say that $f$ has the Laurent property if $x_n$ is a Laurent polynomial of the variables $x_0,\ldots,x_{m-1}$ for all $n\geq 0$. Similarly for $m^\text{th}$-order recurrence
\[
    x_{n+m}=f(x_n,\ldots,x_{n+m-1})
\]
with the initial variables $x_0,\ldots,x_{m-1}$, the recurrence has the Laurent property if $x_n$ can be expressed as a Laurent polynomial in the initial variables $x_0,\ldots,x_{m-1}$ for $n\geq 0$. Sometimes we call it a Laurent recurrence. 

Important examples are given by the Somos and Gale--Robinson recurrences. The Laurent property of generalized Gale--Robinson sequences was established by Fomin and Zelevinsky \cite{FandZ}. %
%
A systematic study of recurrences with the Laurent property was later given by Alman, Cuenca, and Huang \cite{AlmanCuencaHuang}.

The Laurent property is particularly useful for degree growth. If an iterate is written in lowest terms as
$$
x_n=\frac{G(x_0,\ldots,x_{m-1})}
{x_0^{d_n^{(0)}}\cdots x_{m-1}^{d_n^{(m-1)}}},
$$
where $G(x_0,\ldots,x_{m-1})$ is a polynomial in the initial variables, and the numerator and denominator have no common monomial factor, then the denominator exponents can often be determined by a tropical, or max-plus recurrence. This approach was developed for cluster maps by Fordy and Hone \cite{FordyHone2012, Fordy2014}. In particular, the tropical versions of the Somos-$4$ and Somos-$5$ recurrences give the denominator degree sequences of the corresponding iterates. For example, the tropical Somos-$4$ recurrence is
$$
d_{n+4}+d_n
=
\max\{d_{n+3}+d_{n+1},2d_{n+2}\}.
$$
Its degree sequence has quadratic growth, and hence the associated algebraic entropy is zero (see \cite{FordyHone2012,Fordy2014}). More generally, the tropical viewpoint has become a standard tool for studying degree growth in Laurent systems.

The connection between Laurent recurrences and algebraic entropy has also been investigated for reductions of discrete integrable equations. Two-term Gale--Robinson recurrences are one-dimensional reductions of the Hirota--Miwa equation, and three-term ones are reductions of the discrete BKP equation. Mase proved that one-dimensional reductions of the Hirota--Miwa and discrete BKP equations with the Laurent property have zero algebraic entropy \cite{Mase2015}. Hone, Kouloukas, and Ward studied reductions of the Hirota--Miwa equation and showed that these reductions include broad families of Somos and Gale--Robinson recurrences \cite{HoneKouloukasWard2017}. Galashin and Pylyavskyy deduced quadratic degree growth, and hence zero algebraic entropy, for two-term and three-term Gale--Robinson recurrences \cite[Proposition~11.4]{GalashinPylyavskyy} from the combinatorial formulas of Speyer for the octahedron recurrence \cite{Speyer} and of Carroll and Speyer for the cube recurrence \cite{CarrollSpeyer}. These results establish quadratic growth, but, to the best of our knowledge, they do not give explicit closed formulas for the degree sequences.


The aim of this paper is to study the degree growth of a class of recurrences with the Laurent property directly from their recurrence relations. 

We first consider recurrences of the form
\begin{equation}\label{eqn:form-of-rec}
x_{n+m}x_n
=P(x_{n+1},\ldots,x_{n+m-1})=
\sum_{s=1}^{m-1}\sum_{t=1}^{m-1}
c_{s,t}x_{n+s}x_{n+t},
\quad n\geq0,
\end{equation}
where
$$
P\in\mathbb{Z}[x_0,\ldots,x_{m-1}]
$$
is homogeneous of degree two and has non-negative coefficients. Moreover, we assume that $P$ is not divisible by an initial variable $x_i$ for $i=0,\ldots,m-1$.

Our first result gives an explicit formula for the degree of the iterates of the associated rational map.

\begin{theorem}\label{general}
Let $\{x_n\}_{n\geq0}$ be a sequence of rational functions in
$\mathbb{Q}(x_0,\ldots,x_{m-1})$ satisfying \eqref{eqn:form-of-rec}. Suppose that $\{x_n\}_{n\geq0}$ has the Laurent property and that there exists $1\leq t\leq m-1$ such that
$c_{t,m-t}\neq 0$.
Consider the rational map
$$
f\colon\mathbb{A}^m\longrightarrow\mathbb{A}^m,
\qquad
f(x_0,\ldots,x_{m-1})
=
(x_1,\ldots,x_m).
$$
Then
$$
\deg(f^n)
=
1+\left(
\sum_{i=0}^{m-1}
\max\{k_j^{(i)}\mid n\leq j\leq n+m-1\}
\right)
$$
for $0\leq i\leq m-1$ the sequence $\{k_n^{(i)}\}_{n \geq 0}$ satisfies the initial conditions  
$$
k_0^{(i)}=\cdots=k_{m-1+i}^{(i)}=0,
\qquad
k_{m+i}^{(i)}=1,
$$
and of the form
\[
k^{(i)}_{n+m} = \max_{(s,t)\in C}\big(k^{(i)}_{n+s}+k^{(i)}_{n+t}\big) - k^{(i)}_{n}, \quad n\ge i+1,
\]
where $C = \{(s,t) : 1\le s,t\le m-1,\ c_{s,t}\neq 0\}$.
\end{theorem}

The theorem reduces the computation of the degree growth of $f$ to the study of max-linear sequences. This is related to the tropical approach to denominator vectors in cluster dynamics \cite{FordyHone2012,Fordy2014}. Our formulation starts directly from the recurrence \eqref{eqn:form-of-rec} and does not require an a priori description in terms of a cluster seed.

We next consider the particular max-linear recurrence
\begin{equation}\label{2orderrec p=1 q=2}
k_{n+m}
=
\max
\left\{
k_{n+m-1}+k_{n+1}-k_n,
k_{n+m-2}+k_{n+2}-k_n
\right\},
\qquad n\geq1,
\end{equation}
with

$$
k_1=\cdots=k_{m-1}=0,
\qquad
k_m=1.
$$

This is the tropical recurrence associated with the two-term Gale--Robinson recurrence

$$
x_{n+m}x_n
=
x_{n+1}x_{n+m-1}
+
x_{n+2}x_{n+m-2}.
$$

Two-term Gale--Robinson recurrences are classic examples of Laurent recurrences and arise naturally from period-one cluster mutations \cite{FandZ,FordyMarsh2011}.

For Somos-$k$ with $k=4,5,6,7$, Bykovskii, Romanov, and Ustinov \cite{BykovskiiRomanovUstinov2021} explicitly calculated the associated tropical sequences, obtaining quadratic formulas with periodic correction terms of periods $8$, $14$, $20$ and $30$, respectively, by case-by-case computation; see also Nakata \cite{Nakata} for Somos-4 and Somos-5 with arbitrary initial values. They also remarked that there is reason to expect their result to extend to general Gale--Robinson sequences. Theorems~\ref{2order p=1 q=2 Thm} and~\ref{thm on 3order growth rate of GR} independently establish this for two infinite families and recover their results for $m=4,5$ and for $r=3,4$, respectively.

\begin{theorem}\label{2order p=1 q=2 Thm}
Let $m\geq4$ be an integer. Let $\{k_n\}_{n\geq1}$ be the sequence defined by
$$
k_1=\cdots=k_{m-1}=0,
\qquad
k_m=1,
$$
and satisfying \eqref{2orderrec p=1 q=2}. Then
$$
k_n=\frac{n^2+c_n}{2N},
$$
where $\{c_n\}_{n\geq1}$ is a periodic sequence with period
$$
N=(m-2)(m-1)+2.
$$
\end{theorem}


As another application of Theorem \ref{general}, we consider three-term Gale--Robinson recurrences which are of the form

\begin{equation}\label{eqn:three-term-GR}
    x_{n+m}x_n
=
x_{n+p}x_{n+m-p}
+
x_{n+q}x_{n+m-q}
+
x_{n+r}x_{n+m-r}
\end{equation}
for $m=p+q+r$.

%
For the associated rational maps, Corollary~\ref{aeofG-R} gives an expression for the algebraic entropy in terms of the corresponding tropical Gale--Robinson sequence by applying Theorem~\ref{general} to \eqref{eqn:three-term-GR}. 
\begin{corollary}\label{aeofG-R}
Let
$$
f\colon\mathbb{A}^m\longrightarrow\mathbb{A}^m
$$
be the rational map
$$
f(x_0,\ldots,x_{m-1})
=
(x_1,\ldots,x_m),
$$
where $\{x_n\}_{n\geq0}$ is a three-term Gale--Robinson sequence. Then the algebraic entropy of $f$ is 
$$
\mathcal{E}(f)
=
\lim_{n\to\infty}
\frac{1}{n}\log\left(\sum_{i=0}^{m-1}\sum_{j=n}^{n+m-1}k_j^{(i)} \right),
$$
for $0\leq i \leq m-1$ the sequence $\{k_n^{(i)}\}_{n\geq0}$ satisfies
$$
k^{(i)}_{n+m}+k^{(i)}_n
=
\max
\left\{
k^{(i)}_{n+p}+k^{(i)}_{n+m-p},
k^{(i)}_{n+q}+k^{(i)}_{n+m-q},
k^{(i)}_{n+r}+k^{(i)}_{n+m-r}
\right\},
\quad n\geq i+1,
$$
with
$$
k_0^{(i)}=\cdots=k^{(i)}_{m-1+i}=0,
\qquad
k^{(i)}_{m+i}=1.
$$
\end{corollary}

The closed formula in Theorem~\ref{2order p=1 q=2 Thm} and the following conjecture were found independently of \cite{BykovskiiRomanovUstinov2021,Romanov2021}, based on explicit computations of growth rate sequences carried out during the author's master's studies. The later computations in \cite{BykovskiiRomanovUstinov2021} support them: for the Somos-6 and Somos-7 recurrences, that is, for $(p,q,r)=(1,2,3)$ and $(1,2,4)$, they yield quadratic formulas with leading coefficients $1/40$ and $1/60$ and periodic corrections of periods $20$ and $30$, respectively, in agreement with the values $1/(2N)$ and $N$ predicted below; see also \cite{Romanov2021} for further particular Gale--Robinson cases. Now, we formulate the conjecture.

Let $\{k_n\}_{n\geq 1}$ be the sequence of integers such that $$k_1=\ldots=k_{m-1}=0 \text{ and } k_m=1$$ and given by the following recurrence
\begin{equation}\label{3 order k_n}
      k_{n+m}=\max(k_{n+p}+k_{n+m-p},k_{n+q}+k_{n+m-q},k_{n+r}+k_{n+m-r})-k_n
\end{equation} where $1\leq p < q < r$ and $m=p+q+r$. We call this $\{k_n\}_{n\geq 1}$ growth rate sequence of three-term Gale--Robinson. 
    
\begin{conjecture}\label{conj}
    If $\gcd(p,q,r)=1$ then $$k_n=\dfrac{n^2+c_n}{2N}$$ for all $n\geq1$ where $\{c_n\}_{n\geq 1}$ is a periodic sequence with period $$N=\gcd(p,q)r^2+\gcd(p,r)q^2+\gcd(q,r)p^2+pqr.$$ 
\end{conjecture}

The conjecture predicts that the growth rate sequence of every three-term Gale--Robinson recurrence with $\gcd(p,q,r)=1$ is quadratic up to a periodic correction. In view of Corollary~\ref{aeofG-R}, this would give the exact degree growth, refining the known zero-entropy result. Moreover, it gives an explicit candidate for both the quadratic growth coefficient and the period of the correction term in terms of the three parameters.
We prove the conjecture for the following infinite family.

\begin{theorem}\label{thm on 3order growth rate of GR}
Let $\{k_n\}_{n\geq1}$ be the growth rate sequence of the three-term Gale--Robinson recurrence with $p=1$, $q=2$, and $r\geq3$. Then Conjecture~\ref{conj} holds.
\end{theorem}


The results of this paper therefore fit into the broader approach to the study of degree growth of rational maps with Laurent property. Our main contribution is the direct degree formula in Theorem \ref{general} and the explicit analysis of the max-linear recurrence in Theorem \ref{2order p=1 q=2 Thm}. For three-term Gale--Robinson recurrences, we further formulate a conjecture describing the corresponding growth rate sequences by a quadratic term with a periodic correction. We prove this conjecture for the infinite family with parameters $p=1$, $q=2$, and $r\geq3$, providing evidence for the proposed general formula and establishing quadratic growth. To the best of our knowledge, this uniform formula, including the explicit period of the periodic correction term, has not previously been established for this infinite family.

 
%
\section{Preliminaries}

 \subsection{Discrete Valuation Rings}
 We recall the following general definitions. Our main reference is the book \cite{dummit} by Dummit and Foote.
 \begin{definition}
 Let $\mathbb{F}$ be a field. A discrete valuation is a map $v\colon \mathbb{F}\setminus \{0\}\rightarrow \mathbb{Z}$ with the following properties:
 	\begin{enumerate}
 		\item $v$ is surjective,
 		\item $v(xy)=v(x) + v(y)$ for all $x$, $y$ in $\mathbb{F}\setminus \{0\}$,
 		\item $v(x+y)\geq \min(v(x),v(y))$ for all $x$, $y$ in $\mathbb{F}\setminus \{0\}$ with $x+y\neq 0$.
 	\end{enumerate}
 The set $R_v=\{x\in\mathbb{F}\mid v(x)\geq0 \}\cup\{0\}$ is then a subring, called the valuation ring of $v$. An integral domain is called a discrete valuation ring if it is the valuation ring of a discrete valuation $v$. A valuation ring $R_v$ has a unique maximal ideal $M_v:=\{x\in\mathbb{F}\mid v(x)>0 \}\cup \{0\}$.
 \end{definition}
  Every non-zero element $g$ in $\mathbb{F}(x_0,\ldots,x_{m-1})$ can be written uniquely in the form $$g=x_0^n\cdot\dfrac{p}{q},$$ where $n\in \mathbb{Z}$ and $p,q \in \mathbb{F}[x_0,\ldots,x_{m-1}]$ are coprime and both $p$ and $q$ are relatively prime to $x_0$. Then the formula $v_{x_0}(g)=n$ defines a valuation $v_{x_0}$ on the field $\mathbb{F}(x_0,\ldots,x_{m-1})$, called the $x_0$-adic valuation.
 \begin{definition}
 	Let $g$ be a polynomial in $R=\mathbb{F}[x_1^{\pm},\ldots,x_s^{\pm}]$. We know that $g$ can be written as $$g=x_i^{v_{x_i(g)}} g^{(i)}, \qquad 0\leq i \leq m-1,$$ where $v_{x_i}$ is the $x_i$-adic valuation and $v_{x_i}(g^{(i)})=0$. Consider the image $\bar{g}^{(i)}$ of $g^{(i)}$ in the residue field $\mathbb{F}\left(x_0,\ldots,x_{i-1},x_{i+1},\ldots,x_{m-1}\right)$ for $0\leq i \leq m-1$. The image $\bar{g}^{(i)}$ is called the $x_i$-adic image of $g$.
 \end{definition}
 
\begin{example}\label{cancellationex}
  Let $f$ and $g$ be polynomials in $\mathbb{Q}[x_0^\pm,x_1^\pm,x_2^\pm,x_3^\pm]$ given by 
  \begin{align*}
  &f=x_0x_1+x_1x_3+x_0x_2, \\
  &g=x_1x_3-x_0x_2.
  \end{align*} We see that $v_{x_0}(f)=v_{x_0}(g)=0$, $\bar{f}^{(0)}=x_1x_3$ and $\bar{g}^{(0)}=x_1x_3$. It follows that $x_0 \nmid f+g$. On the other hand, $x_1 \mid f+g$, since $\bar{f}^{(1)}+\bar{g}^{(1)}=-x_0x_2+x_0x_2=0.$ 
\end{example}
\begin{definition}
	Consider the set $\mathbb{Z}^{m}$ of $m$-tuples of integers with the lexicographic ordering $\preceq$ defined by the formula $\rho=(\rho_0,\ldots,\rho_{m-1})\preceq\gamma=(\gamma_0,\ldots,\gamma_{m-1})$ if $\rho=\gamma$ or $\rho_i=\gamma_i$ and $\rho_s<\gamma_s$ for some $0\leq s \leq m-1$ and all $0 \leq i <s$.
\end{definition}
\begin{definition}
	We call a set $S\subset \mathbb{Z}^m$ bounded from below if there exists some $c\in \mathbb{Z}$ with $\rho_i\geq c$ for any $\rho=(\rho_0,\ldots,\rho_{m-1})\in S$ and $0\leq i \leq m-1$.
\end{definition}
\begin{definition}
	Let $\rho\in \mathbb{Z}^m$ be an $m$-tuple of integers. We say that $\rho\in S$ is lexicographically minimal in $S$ if $\rho \preceq\gamma$ for all $\gamma \in S$.
\end{definition}
Note that if $S,T\subset \mathbb{Z}^m$ are bounded from below, $\rho$ is the lexicographically minimal element in $S$, and $\tau$ is the lexicographically minimal element in $T$, then the algebraic sum $S+T$ is also bounded from below, and its lexicographically minimal element is $\rho+\tau$.
\begin{definition}
	Let $f\in \mathbb{Z}[x_0^\pm,\ldots,x_{m-1}^\pm]$ be a non-zero Laurent polynomial, and write
	\begin{equation*}
	f=\sum_{\rho \in \mathbb{Z}^{m}}c_{\rho}x^{\rho}. 
	\end{equation*} We define $\mu(f)$ as the coefficient $c_{\rho}$, where $\rho$ is lexicographically minimal in the set $\{ \rho\in \mathbb{Z}^m \mid c_\rho \neq 0 \}$. Additionally, we denote the $\mu(\bar{f}^{(i)})$ by $\mu_i(f)$ where $\bar{f}^{(i)}$ $x_i$-adic image for $f$. 
\end{definition}
We will use the function $\mu$ to show that performing certain operations on Laurent polynomials does not lead to undesired cancellations.

\section{Computing Algebraic Entropy}

\subsection{Proof of Theorem \ref{general}}

\begin{proof}
  
  First, we will show by induction on $n$ that $\mu_i(x_n)>0$ for all $0\leq i\leq m-1$ and all $n\geq0$. It is clear that $\mu_i(x_j)=1$ for $0\leq i,j \leq m-1$. Now, we fix $i\in \{0,\ldots,m-1\}$ and $n\geq0$. Assume that we have $\mu_i(x_j)>0$ for all $j\leq n+m-1$. We will show that $\mu_i(x_{n+m})>0$. By equation \eqref{eqn:form-of-rec}, we have $$\mu_i(x_{n+m}x_n)=\mu_i\left( \sum_{s=1}^{m-1}\sum_{t=1}^{m-1}c_{s,t}x_{n+s}x_{n+t} \right).$$ 
  Now, we look at the $x_i$-adic valuation of the sum $$S=\sum_{s=1}^{m-1}\sum_{t=1}^{m-1}c_{s,t}x_{n+s}x_{n+t}.$$ Let $$\nu=\min\{v_{x_i}(x_{n+s}x_{n+t})\mid 1\leq s,t \leq m-1, c_{s,t}\neq0  \}$$ be the minimal valuation among the terms of the sum $S$. We split the sum $S$ into two summands $S=S_1+S_2$, where $S_1$ is the sum of the terms $c_{s,t}x_{n+s}x_{n+t}$ of valuation $\nu$, and $S_2$ is the sum of the terms of higher valuation. We claim that $v_{x_i}(S_1)=\nu$; indeed, passing to the residue field of the valuation $v_{x_i}$, and looking at the $x_i$-adic images of the terms in $S_1$, we know by induction that the leading terms $\mu_i(c_{s,t}x_{n+s}x_{n+t})$ are strictly positive. It follows that the leading term of the sum of these $x_i$-adic images is also positive, and in particular does not vanish. Hence indeed, $v_{x_i}(S_1)=\nu$ and $\mu_i(S_1)>0$. Since the terms in $S_2$ have higher valuation, we have $v_{x_i}(S)=v_{x_i}(S_1)$ and $\mu_i(S)=\mu_i(S_1)>0$. Hence, $$\mu_i(x_{n+m})\mu_i(x_n)=\mu_i(x_{n+m}x_n)>0.$$ Therefore, $\mu_i(x_{n+m})>0.$ Furthermore, we obtain that 
  \begin{align}\label{xiadicval}
  v_{x_i}(x_{n+m})=\min\{v_{x_i}(x_{n+s})+v_{x_i}(x_{n+t}) \mid c_{s,t}\neq 0  \}-v_{x_i}(x_n),
  \end{align}
  for $1\leq s,t\leq m-1$.
  %
  
Observe that for the initial variables we have $x_i$-adic valuations for $0\leq i,j \leq m-1$
\[
    v_{x_i}(x_j)=\begin{cases}
        0 &\text{ if } j\neq i,\\
        1 &\text{ otherwise}.
    \end{cases}
\]
We claim that for $0\le i\le m-1$
$$
v_{x_i}(x_n)=
\begin{cases}
0 & \text{if } 0\le n\le m+i-1,\ n\neq i,\\
1 & \text{if } n=i,\\
-1 & \text{if } n=m+i.
\end{cases}
$$
Indeed, fix $0\leq i \leq m-1$, and let $n=n'+m$ with $0\leq n'\leq i$. Note that $P$ is not divisible by $x_{i-n'}$. Thus, there is a pair $(s,t)\in C$ with $i-n'\notin \{s,t\}$. Therefore, by induction on $n'$ we get that the minimum valuation in \eqref{xiadicval} is $0$. Hence, $v_{x_i}(x_{n'+m})=-v_{x_i}(x_{n'})$ for $0\leq n'\leq i$ which proves the claim.

Now set $\kappa_n^{(i)}:=-v_{x_i}(x_n)$ for $n\neq i$ and $\kappa_i^{(i)}:=0$. Then
$$\kappa_0^{(i)}=\cdots=\kappa_{m-1+i}^{(i)}=0,\qquad \kappa_{m+i}^{(i)}=1,$$
and, using \eqref{xiadicval} and $\kappa_n^{(i)}:=-v_{x_i}(x_n)$ we obtain
$$\kappa^{(i)}_{n+m} = \max_{(s,t)\in C}\big(\kappa^{(i)}_{n+s}+\kappa^{(i)}_{n+t}\big) - \kappa^{(i)}_{n}
$$
for $n\geq i+1$.

Now, we will observe that $\kappa_n^{(i)}\geq 0$ for all $n\geq 0$. By assumption, $(t,m-t)\in C$, hence
$$\kappa_{n+m}^{(i)}+\kappa_n^{(i)}\geq \kappa_{n+t}^{(i)}+\kappa_{n+m-t}^{(i)}, \qquad n\geq i+1.$$
We show by strong induction on $N$ that $\kappa_N^{(i)}\geq \kappa_{N-t}^{(i)}$ for $N\geq t$ and $\kappa_N^{(i)}\geq \kappa_{N-(m-t)}^{(i)}$ for $N\geq m-t$. If $N\leq m+i$, this is clear from the initial values, since the only nonzero value among $\kappa_0^{(i)},\ldots,\kappa_{m+i}^{(i)}$ is $\kappa_{m+i}^{(i)}=1$. If $N\geq m+i+1$, write $N=n+m$. 
Now assume that for $n'< N$ we have
\[
\kappa^{(i)}_{n'}\geq \kappa^{(i)}_{n'-t}, \text { for } n'\geq t \text{ and } \kappa^{(i)}_{n'}\geq \kappa^{(i)}_{n'-(m-t)} \text{ for } n'\geq m-t.
\]
Thus, we get
$$\kappa_{N}^{(i)}-\kappa_{N-t}^{(i)}\geq \kappa_{n+t}^{(i)}-\kappa_n^{(i)}\geq 0
\quad\text{and}\quad
\kappa_{N}^{(i)}-\kappa_{N-(m-t)}^{(i)}\geq \kappa_{n+m-t}^{(i)}-\kappa_n^{(i)}\geq 0,$$
since $n+t<N$ and $n+m-t<N$. Now, for $N\geq 0$ write $N=qt+r$ with $0\leq r<t$. Then
$$\kappa_N^{(i)}\geq\kappa_{N-t}^{(i)}\geq\cdots\geq\kappa_r^{(i)}=0.$$
Consequently, $v_{x_i}(x_n)\leq 0$ for all $n\neq i$ since $\kappa_n^{(i)}=-v_{x_i}(x_n)$ for $n \neq i$.

Note that $\{x_n\}_{n\geq0}$ has the Laurent property. Therefore, we can write $x_n$ uniquely in the form
\begin{equation}\label{Lproperty}
x_n=\dfrac{g_n}{x_0^{k_n^{(0)}}\cdots x_{m-1}^{k_n^{(m-1)}}}, \qquad n\geq 0,
\end{equation}
where $g_n$ is a polynomial in $x_0,\ldots,x_{m-1}$, the $k_{n}^{(i)}$ are non-negative integers, and for each $0\leq i\leq m-1$ either $k_n^{(i)}=0$ or $x_i\nmid g_n$. In particular,
$$v_{x_i}(x_n)=v_{x_i}(g_n)-k_n^{(i)}, \qquad v_{x_i}(g_n)\geq 0.$$
Now let $n\neq i$. If $k_n^{(i)}>0$, then $x_i\nmid g_n$, so $v_{x_i}(g_n)=0$ and $k_n^{(i)}=-v_{x_i}(x_n)$. If $k_n^{(i)}=0$, then $v_{x_i}(g_n)=v_{x_i}(x_n)\leq 0$. Thus $v_{x_i}(g_n)=0$, which means $x_i\nmid g_n$. Therefore, $k_n^{(i)}=0=-v_{x_i}(x_n)$. Thus, for $n\neq i$ we have $x_i\nmid g_n$ and
$$k_n^{(i)}=-v_{x_i}(x_n)=\kappa_n^{(i)}.$$
For $n=i$ we have $x_i=x_i/1$, so $k_i^{(i)}=0=\kappa_i^{(i)}$. Therefore $k_n^{(i)}=\kappa_n^{(i)}$ for all $n\geq 0$, and the sequences $\{k_n^{(i)}\}_{n\geq0}$ satisfy the initial conditions and the recurrence in the statement.

%
%

Finally, we can compute the degree of iterates of the map $f$. We write the $n^{\text{th}}$ iterate of the map f as $$f^n\colon (x_0,\ldots,x_{m-1})\mapsto (x_n,\ldots,x_{n+m-1}).$$ We denote $m_n^{(i)}:=\max\{k_j^{(i)}\mid n\leq j\leq n+m-1 \}$ for $0\leq i \leq m-1$. In projective coordinates, the map $f^n$ takes the form
	\begin{align*}
	(x_0:\dots:x_{m-1}:y)\mapsto \bigg(&G_nx_0^{(m_n^{(0)}-k_{n}^{(0)})}\cdots x_{m-1}^{(m_{n}^{(m-1)}-k_{n}^{(m-1)})}:\\
	&\dots:\\
	&G_{n+m-1}x_0^{(m_n^{(0)}-k_{n+m-1}^{(0)})}\cdots x_{m-1}^{(m_{n}^{(m-1)}-k_{n+m-1}^{(m-1)})}:\\
	&x_0^{m_n^{(0)}}\cdots x_{m-1}^{m_n^{(m-1)}}y \bigg),
	\end{align*} 
	where $G_j=y^{\deg(g_j)}g_j\left(\frac{x_0}{y},\dots,\frac{x_{m-1}}{y}\right)$  are homogenizations of the polynomials $g_j$'s, and are of degree $$k_j^{(0)}+\dots+k_j^{(m-1)}+1.$$ We claim that the projective coordinates of $f^n$ have no common factor for all $n\geq 0$. Indeed, $y$ can not be a common factor since $G_j$'s are homogenizations of the polynomials $g_j$'s. 
    Now, if $m_n^{(i)}>0$ for any $0\leq i\leq m-1$, then $m_n^{(i)}=k_j^{(i)}>0$ for some $n\leq j \leq n+m-1$. Thus, $x_i \nmid g_j$, so $x_i \nmid G_j$. Now we look at $j^\text{th}$ projective coordinate of the $f^n$ and we see that $$x_i \nmid G_jx_0^{m_n^{(0)}-k_j^{(0)}}\dots x_{i-1}^{m_n^{(i-1)}-k_j^{(i-1)}}x_{i+1}^{m_n^{(i+1)}-k_j^{(i+1)}}\dots x_{m-1}^{m_n^{(m-1)}-k_j^{(m-1)}},$$ since $m_n^{(i)}=k_j^{(i)}$. Additionally, if $m_n^{(i)}=0$, then $x_i \nmid x_0^{m_n^{(0)}}\cdots x_{m-1}^{m_n^{(m-1)}}y$. Hence, $x_i$ can not be a common factor for projective coordinates. 
	Therefore, $$\deg(f^n)=m_n^{(0)}+\cdots+m_n^{(m-1)}+1.$$\qedhere
	\end{proof}

\begin{remark}
If $\gcd(m,t)=1$ in Theorem~\ref{general}, then for every $0\leq i\leq m-1$ there exists $N_i\geq0$ such that $k_n^{(i)}>0$ for all $n\geq N_i$; that is, every initial variable eventually appears in the denominators of the iterates. 
\end{remark}
\begin{proof}
    Note that $\gcd(t,m-t)=\gcd(m,t)=1$. Hence, there exists $N\geq0$ such that every $n\geq N$ can be written as $n=at+b(m-t)$ with $a,b\geq0$. By the monotonicity shown in the proof,
$$k_{n+m+i}^{(i)}\geq k_{m+i}^{(i)}=1, \qquad n\geq N,$$
so one can take $N_i=N+m+i$.
\end{proof}

	\begin{corollary}\label{cor:p=1-gives-max}
		Assume that $c_{s,m-1}\neq 0$ in Theorem \ref{general}. Then $$\deg(f^n)=\left(\sum_{i=0}^{m-1}k_{n+m-1}^{(i)}\right)+1.$$
	\end{corollary}
	\begin{proof}
		We claim that by the assumptions of the Theorem~\ref{general} the sequences $\{k_n^{(i)}\}_{n\geq 0}$ are non-decreasing sequences for all $0\leq i\leq m-1.$
        We will show this by induction on $n$. By equation \eqref{eqn:form-of-rec}, we see that $$k_0^{(i)}=k_1^{(i)}=\cdots=k_{m-1+i}^{(i)}=0,$$ and $k_{m+i}^{(i)}=1$. Now suppose, $k_j^{(i)}\geq k_{j-1}^{(i)}$ for some $n>i$ and all $n\leq j\leq n+m-1$. We will show that $k_{n+m}^{(i)}\geq k_{n+m-1}^{(i)}.$ Since $c_{s,m-1}\neq 0$, we have $$k_{n+m}^{(i)}+k_{n}^{(i)}\geq k_{n+m-1}^{(i)}+k_{n+s}^{(i)}. $$ Then we get $$k_{n+m}^{(i)}\geq k_{n+m-1}^{(i)}+ k_{n+s}^{(i)}-k_{n}^{(i)}$$ By the induction assumption $k_{n+s}^{(i)}-k_{n}^{(i)}\geq 0$. Thus, $$k_{n+m}^{(i)}\geq k_{n+m-1}^{(i)},$$ which completes the inductive proof. Therefore, the result follows then immediately from Theorem \ref{general} since the $m_n^{(i)}=k_{n+m-1}^{(i)}$ for all $n$.
		\end{proof}
	In the following example, we note the necessity of the assumption  $c_{t,m-t}\neq 0$ for some $t$.
	\begin{example}
		Let $\{k_n\}_{n\geq 1}$ be sequence of integers given by $$ k_{n+6}+k_n=\max(2k_{n+1},2k_{n+2})$$
        with the initial values $k_1=\ldots =k_5=0$, and $k_6=1$. 
        One can see that $\{k_n\}_{n\geq 1}$ takes negative integer values.
	\end{example}

\subsection{Max-linearly recurrent sequences}\label{section311}
Note that sometimes we index the sequence $\{k_n\}$ by $n\geq1$ for convenience. This indexing does not necessarily correspond to the indexing of the iterations of the rational map.
\begin{definition}
Call a sequence $\{k_n\}_{n\geq1}$ of real numbers max-linearly recurrent if there exist $d\geq 1$, $s\geq 1$ and real numbers $c_j^{(i)}$ for $1\leq i\leq s$ and $1\leq j \leq d$ such that $$k_n=\max\bigg(\sum_{j=1}^{d}c_j^{(1)}k_{n-j},\sum_{j=1}^{d}c_j^{(2)}k_{n-j},\ldots,\sum_{j=1}^{d}c_j^{(s)}k_{n-j}\bigg), \qquad n\geq d+1.$$
\end{definition}
Every linearly recurrent sequence is max-linearly recurrent. It is an interesting question when a max-linearly recurrent sequence is also linearly recurrent. This is certainly the case if one of the expression $\sum_{j=1}^{d} c_j^{(i_0)}k_{n-j}$ is larger than all $\sum_{j=1}^{d}c_j^{(i)}k_{n-j}$ for some $i_0$ and all sufficiently large $n$. However, this is only a sufficient condition. For a non-trivial example, let $\{k_n\}_{n\geq1}$ be a sequence of non-negative integers with the recurrence 
\begin{equation}\label{mgeq4}
k_{n+m}+k_n=\max(k_{n+m-1}+k_{n+1},k_{n+m-2}+k_{n+2}), \qquad m\geq4. 
\end{equation} Notice that, $\{k_n\}_{n\geq1}$ is a max-linearly recurrent sequence. In the next proof we will show that $\{k_n\}_{n\geq 1}$ is also linearly recurrent sequence for initial values $$k_1=k_2=\dots=k_{m-1}=0 , \text{ and } k_m=1,$$ but in the expression \ref{mgeq4} the maximum realised by both $k_{n+m-1}+k_{n+1}$ and $k_{n+m-2}+k_{n+2}$ for infinitely many $n$. In fact, in this case we are able to give a closed formula for the values $k_n$.
\subsection{Proof of Theorem \ref{2order p=1 q=2 Thm}}
\begin{proof}
	For an integer $1\leq n \leq N-2,$ write $n=(i-1)(m-1)+j$ with $1\leq i \leq m-2$ and $1\leq j \leq m-1$. First, we claim that 
	\begin{equation}\label{formula for k_n p=1 q=2}
	k_n=\binom{i}{2}+\max(j-m+i,0), \qquad 1\leq n \leq N-2.
	\end{equation} The formula is clear for $1\leq n \leq m$. If $n$ is such that $1\leq n \leq N-m-2$, then we put 
	\begin{align*}
	\delta_1=\begin{cases}
	1 & \text{if } i+j \geq m,\\
	0 & \text {otherwise} \end{cases}, \ 
	\delta_2=\begin{cases}1 & \text{if } i+j\geq m-1, \\ 0 & \text{otherwise}\end{cases}, \ 
	\delta_3=\begin{cases}
	1 & \text{if } j=1\\
	0 & \text{otherwise}.
	\end{cases}
    \end{align*} Then a direct computation implies that for all $n$ such that the formula \eqref{formula for k_n p=1 q=2} holds for all the terms of the sequence $\{k_n\}_{n\geq1}$ occurring in the corresponding equations, we have
    \begin{align*}
    &k_{n+1}=k_n+\delta_1,\\
    &k_{n+2}=k_n+\delta_1+\delta_2,\\
    &k_{n+m-1}=k_{n}+i+\delta_1,\\
    &k_{n+m-2}=k_{n}+i-\delta_3,\\
    &k_{n+m}=k_{n}+i+\delta_1+\delta_2.
    \end{align*} Suppose the formula \eqref{formula for k_n p=1 q=2} holds for $k_1,\ldots,k_{n+m-1}$. Then 
    \begin{align*}
    k_{n+m}=&\max(k_{n+m-1}+k_{n+1}-k_n,k_{n+m-2}+k_{n+2}-k_n)\\&=\max(k_{n}+i+2\delta_1,k_n+i+\delta_1+\delta_2-\delta_3)\\&=k_n+i+\delta_1+\delta_2+\max(\delta_1-\delta_2,-\delta_3).
    \end{align*} We have $\delta_1-\delta_2 \in \{0,-1\}$, $-\delta_3\in \{0,-1\}$ and $\max(\delta_1-\delta_2,-\delta_3)=0$, since if we had $$\delta_1-\delta_2=-\delta_3=-1,$$ then $i+j=m-1$, and $j=1$, which would mean that $n=N-m$, a contradiction. Thus $$k_{n+m}=k_n+i+\delta_1+\delta_2$$ and therefore the formula \eqref{formula for k_n p=1 q=2} is also satisfied for $k_{n+m}$, and hence by induction for all $1\leq n \leq N-2$. 
    
    Applying \eqref{formula for k_n p=1 q=2} for $ N-m-1\leq n\leq N-2$ we get that \begin{align}\label{formula delta4}
    k_n=n-\dfrac{N}{2}+\delta_4,
    \end{align} where 
    \begin{align*}
    \delta_4=\begin{cases}
    1 & \text{if } n\in \{N-m-1,N-m\},\\
    0 & \text{otherwise}.
    \end{cases}
    \end{align*}
   Using \eqref{2orderrec p=1 q=2} and \eqref{formula delta4} we get that 
   \begin{equation*}
   k_{N-1}=k_N=\dfrac{N}{2}-1.
   \end{equation*} 
   Thus, for $N-m-1\leq n \leq N$ we have 
   \begin{equation}\label{k_N}
   k_n=n-\dfrac{N}{2}+\delta_5,
   \end{equation}where 
   \begin{align*}
  \delta_5=\begin{cases}
   1 & \text{if } n\in \{N-m-1,N-m\},\\
   -1 & \text{if } n=N,\\
   0 & \text{ otherwise}.
   \end{cases}
   \end{align*}
   Using equations \eqref{2orderrec p=1 q=2} and \eqref{k_N}, it is easy to directly compute the values $k_{N+1},\ldots,k_{N+m}$, obtaining 
   \begin{equation}\label{k_N+n}
   k_{N+n}=\dfrac{N}{2}+n+\delta_6, \qquad 1\leq n \leq m,
   \end{equation} where 
   \begin{align*}
   \delta_6=\begin{cases}
   1 & \text{if } n=m,\\
   0 & \text{otherwise}.
   \end{cases}
   \end{align*}
   Now, we claim that 
	\begin{equation}\label{claim on k_N+n}
	k_{N+n}=\dfrac{N+2n}{2}+k_n
	\end{equation} for $1\leq n \leq N$. Since the form of the recurrence \eqref{2orderrec p=1 q=2} is invariant under changing the sequence by adding to it an arithmetic progression, we see that the sequence $\{l_n\}_{n\geq1}$ given by the formula $$l_n=k_{N+n}-\dfrac{N}{2}-n$$ satisfies the same recurrence \eqref{2orderrec p=1 q=2} and with initial values $l_1=\dots=l_{m-1}=0$ and $l_m=1$. Hence, $l_n=k_n$ for all $n\geq 1$.
	
	Finally, we define the sequence $\{c_n\}_{n\geq 1}$ as $c_n=2Nk_n-n^2$ for all $n\geq 1$. We use the equation \eqref{claim on k_N+n} to check that  
	\begin{align*}
	c_{N+n}&=2Nk_{n+N}-\left(n+N \right)^2=2Nk_n-n^2=c_n
	\end{align*} for all $n\geq 1$. Consequently, we see that $c_n$ is a periodic sequence with period $N$ and $k_n=\dfrac{n^2+c_n}{2N}$.    
\end{proof}

\subsubsection{Somos-4 and Somos-5 Maps}\label{somos4 and somos5}
	
	Let $f$ be the map $$f\colon \mathbb{A}^4 \ni (w,x,y,z)\mapsto (x,y,z,(xz+y^2)/w) \in \mathbb{A}^4 .$$ We consider a sequence of rational functions $\{x_n\}_{n\geq 0}$ with the first four elements $x_0=w$, $x_1=x$, $x_2=y$, $x_3=z$ and satisfy the following recurrence
	\begin{align}\label{somos4rec}
		x_{n+4}x_n=x_{n+3}x_{n+1} + x_{n+2}^2.
	\end{align}
	This recurrence is known as the Somos-4 recurrence, so we call $f$ the Somos-4 map. Similarly, Somos-5 map is a map from $\mathbb{A}^5$ to $\mathbb{A}^5$ which is given by the following recurrence 
	\begin{align}
	x_{n+5}x_n=x_{n+1}x_{n+4}+x_{n+2}x_{n+3}.
	\end{align}

\begin{corollary}\label{Somos-4 ex}
	Let $d_n$ be the degree of the $n^{\text{th}}$ iterate of the Somos-$4$ map. Assume that $\{k_n\}_{n\geq1}$ is a sequence of integers which is given by $k_n=\dfrac{1}{16}(n^2+c_n)$ where $\{c_n\}_{n\geq1}$ is a periodic sequence of integers with period $8$ and values $c_1=-1$, $c_2=-4$, $c_3=-9$, $c_4=0$, $c_5=-9$, $c_6=-4$, $c_7=-1$, $c_8=-16$, $c_9=-1$, $c_{10}=-4$, $c_{11}=-9$, $c_{12}=0$. Then $d_n=k_{n+3}+k_{n+2}+k_{n+1}+k_n+1.$ In particular, the algebraic entropy of the Somos-$4$ map is $0$.
\end{corollary}
\begin{proof}
     It is known that Somos-$k$ sequences satisfy the Laurent property for $k\leq7$ (see \cite{AlmanCuencaHuang,FandZ,hone2008integrality}). By the Somos-4 recurrence we see that it is in the form of the equation \eqref{eqn:form-of-rec}. Thus, we write $$x_{n}=\dfrac{g_n}{x_0^{k_n^{(0)}}x_1^{k_n^{(1)}}x_2^{k_n^{(2)}}x_3^{k_n^{(3)}}}$$ where $g_n$ is a polynomial in $x_0,x_1,x_2,x_3$ and $k_n^{(i)}$ are some nonnegative integers for $0\leq i \leq 3$. If $k_n^{(i)}>0$ then $x_i$ does not divide $g_n$ for $0\leq i \leq 3$. 
     
     Using the Somos-4 recurrence \eqref{somos4rec} and Theorem \ref{general} we conclude that $$d_n=k_{n+3}^{(0)}+k_{n+3}^{(1)}+k_{n+3}^{(2)}+k_{n+3}^{(3)}+1, \qquad n\geq0,$$ where
     \begin{equation}\label{eqinSomos4}
     k_{n+4}^{(i)}+k_n^{(i)}=\max(k_{n+3}^{(i)}+k_{n+1}^{(i)},2k_{n+2}^{(i)}), \qquad 0\leq i \leq 3,\end{equation}
      and with the initial values $$k_{1}^{(i)}=k_2^{(i)}=\cdots=k_{3+i}^{(i)}=0, \ k_{4+i}^{(i)}=1 \text{ for } 0\leq i \leq 3.$$ Using equation \eqref{eqinSomos4} and initial values, it is easy to see that $$k_{n+3}^{(0)}=k_{n+4}^{(1)}=k_{n+5}^{(2)}=k_{n+6}^{(3)}.$$
      Therefore, $d_n=k_{n+3}^{(0)}+k_{n+2}^{(0)}+k_{n+1}^{(0)}+k_{n}^{(0)}+1.$
      
      Now, we see that $\{k_n^{(0)}\}_{n\geq 1}$ in the required form of the assumptions of Theorem \ref{2order p=1 q=2 Thm}. In this case $m=4$, we compute $N=8$ and we get that
      $$k_n^{(0)}=\dfrac{n^2+c_n}{16}$$ where $\{c_n\}_{n\geq 1}$ is a periodic sequence with period $8$. Finally, one can easily compute the initial values for $\{c_n\}_{n\geq 1}$. Hence, degree sequence of iterates of Somos-4 map growth like $n^2$, so that the algebraic entropy of the Somos-4 map is $0$.
\end{proof}	
\begin{corollary}
		Let $d_n$ be the degree of the $n^{\text{th}}$ iterate of the Somos-$5$ map. Assume that $\{k_n\}_{n\geq1}$ is a sequence of integers which is given by $k_n=\dfrac{n^2+c_n}{28}$ where $\{c_n\}_{n\geq1}$ is a periodic sequence of integers with period $14$. Then $$d_n=k_{n+4}+k_{n+3}+k_{n+2}+k_{n+1}+k_n+1.$$ In particular, the algebraic entropy of the Somos-$5$ map is $0$.
\end{corollary}
\begin{proof}
	The proof is analogous to the Somos-4 case, so we only explain the main steps.
	Somos-5 map is defined as 
	$$f \colon \mathbb{A}^5 \ni (w,x,y,z,t)\mapsto (x,y,z,t,(xt+yz)/w) \in \mathbb{A}^5.$$ We consider a sequence of rational functions $\{x_n\}_{n\geq0}$ by the recurrence $$x_{n+5}x_n=x_{n+1}x_{n+4} + x_{n+2}x_{n+3},$$ setting $x_0=w$, $x_1=x$, $x_2=y$, $x_3=z$ and $x_4=t$. Again, by the Laurent property of Somos-k sequences for $k\leq 7$, we directly write that 
	\begin{align*}
	x_n=\dfrac{g_n}{w^{k_n}x^{l_n}y^{m_n}z^{r_n}t^{s_n}}
	\end{align*} where $g_n\in \mathbb{Z}[x_0,x_1,x_2,x_3,x_4]$ and $k_n$, $l_n$, $m_n$, $r_n$ and $s_n$ are nonnegative integers for all $n$ and we may assume that $x_0$, $x_1$, $x_2$, $x_3$ and $x_4$ do not divide $g_n$ if $k_n>0$, $l_n>0$, $m_n>0$, $r_n>0$ and $s_n>0$, respectively. Using the same arguments in Somos-4 we can write that $k_{n+5}=\max (k_{n+4}+k_{n+1},k_{n+2}+k_{n+3})-k_n$ and $$d_n=k_{n+4}+k_{n+3}+k_{n+2}+k_{n+1}+k_n+1.$$
	
	Since $m=5$, we find that $k_n=\frac{1}{28}(n^2+c_n)$ where $\{c_n\}_{n\geq 1}$ is a periodic sequence with period $14$. Hence the algebraic entropy of the Somos-5 map is 0.  
\end{proof}
\begin{corollary}
	Let $f$ be a rational map of the form $f\colon (x_0,\ldots,x_{m-1})\mapsto (x_1,\ldots,x_m)$ where $\{x_n\}_{n\geq 0}$ is a sequence of rational functions which satisfy the following recurrence $$x_{n+m}x_n=x_{n+m-1}x_{n+1}+x_{n+m-2}x_{n+2}$$ for any $m\geq 4$. Then the algebraic entropy of $f$ is 0.
\end{corollary}

\section{Growth Rate Sequences of Gale--Robinson}

\begin{definition}\label{definitiongale}
	We say that a sequence $\{x_n\}_{n\geq0}$ is a three-term Gale--Robinson if it satisfies the following recurrence $$x_{n+m}x_n=Ax_{n+p}x_{n+m-p}+Bx_{n+q}x_{n+m-q}+Cx_{n+r}x_{n+m-r}$$ where $1\leq p<q<r$ are integers and $m=p+q+r$ and $A,B,C$ are positive integer coefficients.   
\end{definition}
\begin{corollary}\label{cor:three-term-gr-p=1}
Let
$$
f\colon\mathbb{A}^m\longrightarrow\mathbb{A}^m
$$
be the rational map
$$
f(x_0,\ldots,x_{m-1})
=
(x_1,\ldots,x_m),
$$
where $\{x_n\}_{n\geq0}$ is a three-term Gale--Robinson sequence with $p=1$. Then the algebraic entropy of $f$ is 
\[
\mathcal{E}(f)
=
\lim_{n\to\infty}
\frac{1}{n}\log\left(\sum_{i=0}^{m-1}k_{n+m-1}^{(i)} \right),
\]
the sequence $\{k_n^{(i)}\}_{n\geq0}$ satisfies
$$
k^{(i)}_{n+m}+k^{(i)}_n
=
\max
\left\{
k^{(i)}_{n+p}+k^{(i)}_{n+m-p},
k^{(i)}_{n+q}+k^{(i)}_{n+m-q},
k^{(i)}_{n+r}+k^{(i)}_{n+m-r}
\right\},
\quad n\geq i+1,
$$
with
$$
k_0^{(i)}=\cdots=k^{(i)}_{m-1+i}=0,
\qquad
k^{(i)}_{m+i}=1.
$$
for $0\leq i \leq m-1$.
    
\end{corollary}
\begin{definition}
       Let $\{k_n\}_{n\geq 1}$ be the sequence of integers such that $$k_1=\ldots=k_{m-1}=0 \text{ and } k_m=1$$ and given by the following recurrence
      \begin{equation}\label{3 order k_n}
      k_{n+m}=\max(k_{n+p}+k_{n+m-p},k_{n+q}+k_{n+m-q},k_{n+r}+k_{n+m-r})-k_n
      \end{equation} where $1\leq p < q < r$ and $m=p+q+r$. The sequence $\{k_n\}_{n\geq 1}$ is called the growth rate sequence of the three-term Gale--Robinson. 
\end{definition}
%
%

%


      %
\subsection{Conjecture On Growth Rate Sequence}\label{conjecturesec}
     In this section we focus on the Conjecture~\ref{conj}.
     Note that in previous sections we already showed a similar result for the two-term case when $p=1$, $q=2$ and $m\geq 4$ in Theorem \ref{2order p=1 q=2 Thm}. Let us recall the Conjecture.
    \newline
    \textbf{Conjecture:}
        If $\gcd(p,q,r)=1$, then $$k_n=\dfrac{n^2+c_n}{2N}$$ for all $n\geq1$ where $\{c_n\}_{n\geq 1}$ is a periodic sequence with period $$N=\gcd(p,q)r^2+\gcd(p,r)q^2+\gcd(q,r)p^2+pqr.$$ 
   
 Now, before proving the Conjecture~\ref{conj} for the specific class of growth rate sequences $\{k_n\}_{n\geq 1}$ satisfying \eqref{3 order k_n}, we need some preliminary lemmas. In the first lemma, we will see a combinatorial formula for this specific class of growth rate sequences $\{k_n\}_{n\geq 1}$.
   \begin{lemma}\label{Lemma on f(i,j,m)}
   	Let $\{k_n\}_{n\geq 1}$ be the growth rate sequence of three-term Gale--Robinson with $p=1$, $q=2$ and $r\geq 3$. For $1\leq n \leq N/2$ write $n$ in the form $n=(i-1)(m-1)+j$ for $1\leq i$ and $1\leq j \leq m-1$, and define $$f(i,j,m)\footnotemark=f(n):=\binom{i}{2}+\max\left(\ceil[\bigg]{\dfrac{j+2(i-1)-(m-1)}{2}},0\right).$$
    \footnotetext{Since the triple $(i,j,m)$ determines the n we note the function by $f(n)$ if $(i,j,m)$ clear by the context}
    Then the following equations hold. 
        \begin{enumerate}
        
   		\item[i)] $k_n=f(n)$ if $1\leq n<N/2$, 
   		\item[ii)] $k_{N/2}=f(N/2)$ if $m$ is even, 
   		\item[iii)] $k_{N/2}=f(N/2)-1$ if $m$ is odd.
   	\end{enumerate} 
   		\end{lemma}
   \begin{proof}
   	Obviously, we see that $k_n=f(n)$ for $1\leq n \leq m$. If $n$ is such that $1\leq n < N/2-m$, then we set the followings;
   	\begin{align*}
   	\delta_1=\begin{cases}
   	1 & \text{if }  j\geq m-2i\\
   	& \text{and } 2 \nmid m+j,\\
   	0 & \text {otherwise} \end{cases}, \ 
   	\delta_2=\begin{cases}
   	1 & \text{if } j\geq m-2i, \\ 
   	0 & \text{otherwise}\end{cases}, \ 
   	\delta_3= \begin{cases}
   	1 & \text{if } m-2i\leq j\leq m-3\\ 
   	  & \text{and }2\nmid m+j, \\
   	0 & \text{otherwise} 
   	\end{cases},
   	\end{align*}  
   	\begin{align*}
   	 \delta_4=\begin{cases}
   	 i-1 & \text{if }  j=1,\\
   	 i & \text {otherwise} \end{cases}, \ 
   	 \delta_5=\begin{cases}
   	 i-1 & \text{if } j\leq 2, \\ 
   	 i & \text{otherwise}\end{cases}, \
   	\delta_6= \begin{cases}
   	1 & \text{if } j\geq m-2i-1,\\ 
   	0 & \text{otherwise}. 
   	\end{cases}
   	\end{align*}
   	Then we see that the formula $f(n)$ gives the following equivalent formulas:
   	\begin{align*}
   	&f(n+1)=f(n)+\delta_1,\\
   	&f(n+2)=f(n)+\delta_2,\\
   	&f(n+r)=f(n)+\delta_5,\\
   	&f(n+m-1)=f(n)+i+\delta_2,\\
   	&f(n+m-2)=f(n)+\delta_4+\delta_1,\\
   	&f(n+m-r)=f(n)+\delta_3+\delta_6,\\
   	&f(n+m)=f(n)+i+\delta_1+\delta_6.
   	\end{align*} 
   	for the terms of the sequence $\{k_n\}_{n\geq1}$ for all $1\leq n < N/2-m$. Suppose the formula $f(n)$ holds for $k_1,\ldots,k_{n+m-1}$. Then 
   	\begin{align*}
   	k_{n+m}&=\max(k_{n+1}+k_{n+m-1}-k_n,k_{n+2}+k_{n+m-2}-k_n,k_{n+r}+k_{n+m-r}-k_n)\\
   	&=k_n+\max(\delta_1+i+\delta_2,\delta_1+\delta_2+\delta_4,\delta_3+\delta_5+\delta_6)\\
   	&=k_n+\begin{cases}
   	i+2 & \text{if } j\geq m-2i \text{ and } 2\nmid m+j,\\
   	i+1 & \text{if } j=m-2i-1,\\
   	i+1 & \text{if } j\geq m-2i-1 \text{ and } 2\mid m+j,\\
   	i   & otherwise.
   	\end{cases}\\
   	&=k_n+i+\delta_1+\delta_6.
   	\end{align*}
   	Thus, the formula $f(n)$ is also true for $k_{n+m}$ and therefore by induction for all $1\leq n < N/2$. This completes the proof of the first claim. 
   	
    Therefore, direct computations show the second and third claim using the recurrence of the sequence $\{k_n\}_{n\geq1}$ and the formula $f(n)$.
   	\end{proof}
     \begin{lemma}\label{Lemma 2 on p1 q2 r arbit}
     	Let $\{k_n\}_{n\geq 1}$ be the growth rate sequence of three-term Gale--Robinson with $p=1$, $q=2$ and $r\geq 3$. Then $k_{N/2+n}=k_{N/2-n}+n$ where $1\leq n < N/2.$ Moreover, $k_{N+n}=\frac{N+2n}{2}+k_n$ for $1\leq n\leq N$.
     \end{lemma}
    \begin{proof}
    	First, we want to prove that 
    	\begin{equation}\label{eqn k_N/2}
    	k_{N/2+n}=k_{N/2-n}+n, \qquad 1\leq n< N/2.
    	\end{equation}
    	We prove this by induction on $n$. In order to use induction, we need to check that formula \eqref{eqn k_N/2} holds for $1\leq n \leq m$. Assume that $m=p+q+r$ is an even integer, which means $r$ is odd. Then we compute that $$N=p^2+q^2+r^2+pqr=(m-1)(m-3)+5,$$ which implies the following 
    	\begin{equation}\label{eqn N/2}
    	N/2=(m-1)\left(\dfrac{m-4}{2}\right) + \dfrac{m+4}{2}.
    	\end{equation}
    	Using equation \eqref{eqn N/2} we write the following equations;
    	\begin{align*}
    	&N/2-m+1=(m-1)\left(\dfrac{m-6}{2}\right)+\dfrac{m+4}{2},\\
    	&N/2-m+2=(m-1)\left(\dfrac{m-6}{2}\right)+\dfrac{m+6}{2},\\
    	&N/2-m+3=(m-1)\left(\dfrac{m-6}{2}\right)+\dfrac{m+8}{2},\\
    	&N/2-m+r+1=(m-1)\left(\dfrac{m-4}{2}\right)+\dfrac{m}{2},\\
    	&N/2-1=(m-1)\left(\dfrac{m-4}{2}\right)+\dfrac{m+2}{2},\\
    	&N/2+1-r=(m-1)\left(\dfrac{m-6}{2}\right)+\dfrac{m+10}{2}.\\
    	\end{align*}
    	By the recurrence of the $\{k_n\}_{n\geq 1}$ we have that 
    	\begin{align*}
    	k_{N/2+1}&=max\left(k_{N/2+2-m}+k_{N/2},k_{N/2-m+3}+k_{N/2-1},k_{N/2-m+r+1}+k_{N/2+1-r}\right) \\ &-k_{N/2-m+1}.
    	\end{align*}
        Now we use the equations above and the Lemma \ref{Lemma on f(i,j,m)}. To be rigorous, we assume that $m\geq 12$. Notice that if $m<12$, then we need to arrange some of the equations above to have $1\leq j \leq m-1$ for the formula $f(n)$. Moreover, one can observe that for $m\in \{6,8,10\}$ the formula in \eqref{eqn k_N/2} holds via direct calculations.
        Therefore,

\begin{align*}
k_{N/2+1}
=\max\Big[f\Big(\tfrac{m-4}{2},\tfrac{m+6}{2},m\Big)+f\Big(\tfrac{m-2}{2},\tfrac{m+4}{2},m\Big),\\
f\Big(\tfrac{m-4}{2},\tfrac{m+8}{2},m\Big)+f\Big(\tfrac{m-2}{2},\tfrac{m+2}{2},m\Big),\\
f\Big(\tfrac{m-2}{2},\tfrac{m}{2},m\Big)+f\Big(\tfrac{m-4}{2},\tfrac{m+10}{2},m\Big)\Big]\\
-f\Big(\tfrac{m-4}{2},\tfrac{m+4}{2},m\Big)\\
=\binom{(m-2)/2}{2}+\Big\lceil\tfrac{m-4}{4}\Big\rceil+1=k_{N/2-1}+1.
\end{align*}
  
Now, assume that $m$ is odd, so that $r=m-3$ is even. Then $\gcd(q,r)=2$. Thus,
$$N=r^2+q^2+2p^2+pqr=r^2+2r+6=(m-1)(m-3)+6.$$
\begin{equation}\label{eqn:N/2-when-m-is-odd}
    N/2=(m-1)\Big(\frac{m-3}{2}\Big)+3 
\end{equation}
Now we use Lemma~\ref{Lemma on f(i,j,m)}(iii) we have $k_{N/2}=f(N/2)-1$. We set
\[
a=\dfrac{m-3}{2} \text{ and } b=\dfrac{m-1}{2}.
\]
We have $N/2-m+r+1=N/2-2$ and $N/2+1-r=N/2-m+4$ since $r=m-3$. Moreover,
\begin{align*}
N/2-m+1&=(m-1)(a-1)+3,\\  
N/2-m+2&=(m-1)(a-1)+4,\\
N/2-m+3&=(m-1)(a-1)+5,\\
N/2-m+4&=(m-1)(a-1)+6,\\
N/2-2&=(m-1)(b-1)+1,\\
N/2-1&=(m-1)(b-1)+2.
\end{align*}
By Lemma~\ref{Lemma on f(i,j,m)} we obtain
\begin{align*}
    &k_{N/2-m+1}=k_{N/2-m+2}=\binom{a}{2},\\ &k_{N/2-m+3}=k_{N/2-m+4}=\binom{a}{2}+1,\\ &k_{N/2-2}=k_{N/2-1}=k_{N/2}=\binom{b}{2}.
\end{align*}
Hence
\begin{align*}
k_{N/2+1}&=\max\big(k_{N/2-m+2}+k_{N/2},\;k_{N/2-m+3}+k_{N/2-1},\;k_{N/2-2}+k_{N/2-m+4}\big)-k_{N/2-m+1}\\
&=\max\Big(\binom{b}{2},\,\binom{b}{2}+1,\,\binom{b}{2}+1\Big)=\binom{b}{2}+1=k_{N/2-1}+1.
\end{align*}
Note that here all indices satisfy $1\le j\le m-1$ for every odd $m\ge7$, so no small cases need separate treatment.

The remaining base cases $2\leq n\leq m$ follow from analogous computations using Lemma~\ref{Lemma on f(i,j,m)} and the decompositions above. For $m$ even, they are based on \eqref{eqn N/2} and $k_{N/2}=f(N/2)$. For $m$ odd, they are based on \eqref{eqn:N/2-when-m-is-odd} and $k_{N/2}=f(N/2)-1$.

Therefore, we have equation~\eqref{eqn k_N/2} holds for $1\leq n \leq m$. Suppose that equation \eqref{eqn k_N/2} holds for all $t\leq n+m-1$ with $1\leq n<N/2-m$. Now, we will show that equation \eqref{eqn k_N/2} holds for $t=n+m$. We know that
        \begin{align*}
        k_{N/2+n+m}=&\max(k_{N/2+n+1}+k_{N/2+n+m-1},k_{N/2+n+2}+k_{N/2+n+m-2},\\
        &k_{N/2+n+r}+k_{N/2+n+m-r})-k_{N/2+n}.
        \end{align*} We use the induction assumption in the previous equation, and we get 
        \begin{align*}
        k_{N/2+n+m}=&\max(k_{N/2-n-1}+n+1+k_{N/2-n-m+1}+n+m-1,\\
        &k_{N/2-n-2}+n+2+k_{N/2-n-m+2}+n+m-2,\\
        &k_{N/2-n-r}+n+r+k_{N/2-n-m+r}+n+m-r)-(k_{N/2-n}+n)\\
        &=k_{N/2-(n+m)}+n+m.
        \end{align*}
        Moreover, using equation \eqref{eqn k_N/2} and the recurrence of the $\{k_n\}_{n\geq 1}$ we directly compute that $$k_{N+n}=\dfrac{N+2n}{2}+k_n, \qquad 1\leq n \leq m.$$ 
        We see that the sequence $\{l_n\}_{n\geq1}$ is given by $$l_n=k_{N+n}-\dfrac{N+2n}{2}$$ satisfies the recurrence of $\{k_n\}_{n\geq1}$ and with initial values $l_1=\dots=l_{m-1}=0$ and $l_m=1$. Hence, $l_n=k_n$ for all $n\geq 1$.
    \end{proof}
    %
      %
\subsection{Proof of Theorem \ref{thm on 3order growth rate of GR}}
      \begin{proof}
      	We define a sequence $\{c_n\}_{n\geq1}$ such that $$c_n=2Nk_n-n^2$$ for all $n\geq 1$. In order to see that $\{c_n\}_{n\geq1}$ is a periodic sequence with period $N$. We use the Lemma \ref{Lemma 2 on p1 q2 r arbit} and check the following
      	\begin{align*}
      	c_{N+n}=2Nk_{N+n}-\left(n+N\right)^2=2N\left(\dfrac{N+2n}{2}+k_n\right)-\left(N+n\right)^2=2Nk_n-n^2=c_n
      	\end{align*} for all $n\geq 1$. 
      	\end{proof}
\begin{corollary}
The Somos-6 and Somos-7 recurrences are the three-term Gale--Robinson recurrences with $(p,q,r)=(1,2,3)$ and $(1,2,4)$, respectively. Hence, by Corollary~\ref{cor:p=1-gives-max} and Theorem~\ref{thm on 3order growth rate of GR}, the degrees of the iterates of the corresponding maps are given by 
$$d_n = 1+\sum_{i=0}^{m-1} k^{(i)}_{n+m-1} = 1+\sum_{j=n}^{n+m-1} k_j, \qquad n\geq 1,$$
where for $0\leq i \leq m-1$
$$k_n^{(i)}=k_{n-i} \text{ for } n\geq i+1, \qquad k_n=\frac{n^2+c_n}{2N},$$
with $m=6$, $N=20$ and $m=7$, $N=30$, respectively, where $\{c_n\}$ is periodic with period $N$. In particular, both maps have zero algebraic entropy.
\end{corollary}

      	\section*{Acknowledgments}
        The results presented in this paper were originally obtained during the author's master's studies. The author would therefore like to thank his master's thesis supervisor, Jakub Byszewski, for his guidance at that time and for his helpful suggestions on the introduction and the overall structure of the present paper. The author (M.E. Can) was supported by the Croatian Science Foundation under the project number HRZZ-UIP-2025-02-4309.
      	
      	\bibliographystyle{amsplain}
      	\bibliography{GrowthRateSequences}	
\end{document}